\documentclass[12pt]{amsart}

\usepackage{amssymb}

\usepackage{enumerate}

\makeatletter
\@namedef{subjclassname@2010}{%
  \textup{2010} Mathematics Subject Classification}
\makeatother

\usepackage{amsmath}
\usepackage{amsthm}

\usepackage{xy}

\usepackage{xypic}

\newtheorem{main-theorem*}{Theorem}
\newtheorem{lemma*}[main-theorem*]{Lemma}

\renewcommand{\mod}{\operatorname{mod}}

\newcommand{\rad}{\operatorname{rad}}

\newcommand{\op}{\operatorname{op}}

\newcommand{\bA}{\mathbb{A}}
\newcommand{\bB}{\mathbb{B}}
\newcommand{\bC}{\mathbb{C}}
\newcommand{\bD}{\mathbb{D}}
\newcommand{\bE}{\mathbb{E}}
\newcommand{\bF}{\mathbb{F}}
\newcommand{\bG}{\mathbb{G}}

\newcommand{\bL}{\mathbb{L}}

\newcommand{\FF}{\mathbb{F}}

\begin{document}

\baselineskip=17pt

\title{Deformed mesh algebras of Dynkin type $\mathbb{F}_4$}

\author[J. Bia\l kowski]{Jerzy Bia\l kowski}
\address{Faculty of Mathematics and Computer Science,
   Nicolaus Copernicus University,
   Chopina~12/18,
   87-100 Toru\'n,
   Poland}
\email{jb@mat.uni.torun.pl}

\date{}



\begin{abstract}
We prove that every deformed mesh algebra of type $\FF_4$
is isomorphic to the canonical mesh algebra of type $\FF_4$.
\end{abstract}

\maketitle

\section*{Introduction}

Throughout this article, $K$ will denote a fixed algebraically 
closed field.
By an algebra we mean an associative finite-dimensional $K$-algebra
with  identity, which we moreover assume to be basic and connected.
For an algebra $A$, we denote by $\mod A$ the category of
finite-dimensional right $A$-modules and by $\Omega_A$ the syzygy
operator which assigns to a module $M$ in $\mod A$ the kernel
of a minimal projective cover $P_A(M) \to M$ of $M$ in $\mod A$.
Then a module $M$ in $\mod A$ is called \emph{periodic} 
if $\Omega_A^n(M) \cong M$ for some $n \geq 1$.
Further, the category of finite-dimensional $A$-$A$-bimodules over
an algebra $A$ is canonically equivalent to the module category
$\mod A^e$ over the enveloping algebra $A^e = A^{\op} \otimes_K A$ of $A$.
Then an algebra $A$ is called a \emph{periodic algebra} 
if $A$ is a periodic module in $\mod A^e$.
It is known that if $A$ is a periodic algebra then is self-injective 
and  
every module $M$ in $\mod A$ without non-zero projective 
direct summands is periodic.
Periodic algebras play currently a prominent r\^ole in the representation
theory of algebras and have attracted much attention (see the survey
article \cite{ESk}).
In particular, it has been proved in \cite{Du} that
all self-injective algebras of finite representation type
(different from $K$) are periodic.
We refer also to recent articles \cite{ESk2,ESk3,ESk4}
on the connections of periodic algebras with finite groups
and triangulated surfaces.

In this note we are concerned with the classification 
of deformed mesh algebras of Dynkin types 
$\bA_n (n \geq 2)$, 
$\bB_n (n \geq 2)$, 
$\bC_n (n \geq 3)$, 
$\bD_n (n \geq 4)$, 
$\bE_6$, $\bE_7$, $\bE_8$, 
$\bF_4$, $\bG_2$, 
and $\bL_n (n \geq 1)$.
It was shown in \cite{BES1,ESk}
that for these algebras
the third syzygy permutes the isomorphism classes of simple modules.
It was also shown in \cite{BES1,ESk} that, for algebraically 
closed fields of positive characteristic, the deformed
mesh algebras of Dynkin type
are periodic algebras.
On the other hand, for algebraically closed fields
of characteristic $0$, it is expected that
every deformed mesh algebra of Dynkin type 
is isomorphic to the mesh algebra of Dynkin type, 
and hence is also a periodic algebra.
The class of deformed mesh algebras of Dynkin type 
contains the deformed preprojective algebras of 
generalized Dynkin types
$\bA_n (n \geq 2)$, 
$\bD_n (n \geq 4)$, 
$\bE_6$, $\bE_7$, $\bE_8$ 
and $\bL_n (n \geq 1)$, 
which occur naturally in very different contexts.
For these, the third syzygy  of any simple module is isomorphic to its
shift by the Nakayama functor (see \cite{BES1}).
We refer to \cite{BES1,ESk},
for results on the importance of these algebras
in the representation theory of self-injective algebras.
Mesh algebras of Dynkin types include in particular 
the  stable Auslander algebras of the Arnold's simple hypersurface
singularities \cite{Arn}. 
In fact, it is an interesting open problem 
whether 
any deformed mesh algebra, of Dynkin type 
over an arbitrary closed field $K$
is a 
stable Auslander algebra of a simple hypersurface singularity.
We recall that 
it was proved in \cite{BES2} that the
deformed preprojective algebras of generalized Dynkin type 
$\bL_n$ (in the sense of \cite{BES1})
are exactly (up to isomorphism) the stable Auslander algebras of
simple plane singularities of Dynkin type $\bA_{2n}$.
Moreover, 
it was shown in \cite{BES3} that the
deformed mesh algebras of Dynkin type $\bC_n$ are isomorphic 
to the canonical mesh algebras of type $\bC_n$, and hence
to the stable Auslander algebras of the unique simple plane curve
singularity of type $\bA_{2n-1}$.
It was also shown in \cite{B2} that
deformed mesh algebras of Dynkin type $\bE_6$ are isomorphic 
to the canonical mesh algebra of type $\bE_6$.
It is known, that the mesh algebra of Dynkin type $\bG_2$
is in fact a tubular algebra 
(isomorphic to the algebra $A_4$ from \cite{BS1}),
and it follows from proof of \cite[Lemma~5.16]{BS1} 
that
the deformed mesh algebras of Dynkin type $\bG_2$ are isomorphic 
to the canonical mesh algebra of type $\bG_2$.
On the other hand, there exist 
deformed mesh algebras of the Dynkin types 
$\bB_n$ (see \cite[Example~9.1]{ESk}), 
$\bD_n$ (see \cite[Proposition~6.1]{BES1}), 
$\bE_7$ and $\bE_8$ (see \cite[Theorem]{B1})
not isomorphic 
to the canonical mesh algebras of these types.
But classifying these algebras seems to be a difficult problem.
For more information 
on classification and periodicity of mesh algebras
we refer to \cite{Du2}.
For more information 
on hypersurface singularities 
we refer to \cite{GK,KS}.
For general background on the
representation theory 
and
selfinjective algebras
we refer to \cite{SY}.

\section*{Results}

The main aim of this article is to prove the following theorem 
providing the classification of deformed mesh algebras of type $\FF_4$.

\begin{main-theorem*}
\label{thm}
    Every deformed mesh algebra of type $\FF_4$
    is isomorphic to the canonical mesh algebra of type $\FF_4$.
\end{main-theorem*}

We recall that the 
\emph{canonical mesh algebra $\Lambda(\bF_4)$ of type $\FF_4$} 
is given by the quiver
$$
    \begin{array}{c} Q_{\mathbb{F}_{4}}: \\ \end{array}
   \quad \vcenter
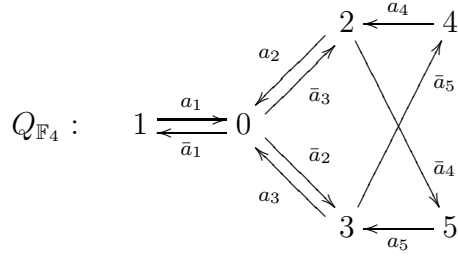
{
    \xymatrix@C=.8pc@R=2pc{
        &&&& 2  \ar@<-.5ex>[lld]_{a_2} \ar[rrdd]^(.75){\bar{a}_4}
        && 4 \ar[ll]_{a_4}
        \\
         1 \ar@<.5ex>[rr]^{a_1} & & 0 \ar@<.5ex>[ll]^{\bar{a}_1}
           \ar@<-.5ex>[rru]_{\bar{a}_3} \ar@<.5ex>[rrd]^{\bar{a}_2}
        \\
        &&&& 3 \ar@<.5ex>[llu]^{a_3} \ar[rruu]_(.75){\bar{a}_5}
        && 5 \ar[ll]^{a_5}
        \\
    }
   }
$$
and the relations
\begin{gather*}
  a_1 \bar{a}_1 = 0, 
  \quad 
 \bar{a}_1 a_1 + \bar{a}_3 a_2 + \bar{a}_2 a_3 = 0,
 \\
  {a}_2 \bar{a}_2 + \bar{a}_4 {a}_5 = 0,  
  \quad 
  {a}_3 \bar{a}_3 + \bar{a}_5 {a}_4 = 0,  
  \quad 
  {a}_4 \bar{a}_4 = 0,  
  \quad 
  {a}_5 \bar{a}_5 = 0.  
\end{gather*}
We note that $\Lambda(\FF_4)$ is not a weakly symmetric algebra, 
and hence not a symmetric algebra.
Further, consider the local commutative algebra
$$
   R(\mathbb{F}_4) = K \langle x, y \rangle /
   \left( xyx, yxy, (x+y)^{2} \right)
   ,
$$
which is isomorphic to the algebra $e_0 \Lambda(\FF_4) e_0$, where $e_0$
is the primitive idempotent in $\Lambda(\FF_4)$ associated to the vertex
$0$ of $Q_{\FF_4}$.
An element $f$ from the square $\rad^2 R(\FF_4)$ of the radical
$\rad \ R(\FF_4)$ of $R(\FF_4)$ 
is said to be \textit{admissible}
if $f$ satisfies the following condition
$$\big(x+y + f(x,y)\big)^{2} = 0.$$
Let $f \in {\rm rad}^2R(\FF_4)$ be admissible.
We denote by $\Lambda^f(\FF_4)$ 
the algebra  given by the quiver $Q_{\FF_4}$ and the relations
\begin{gather*}
  a_1 \bar{a}_1 = 0, 
  \quad 
  \bar{a}_2 a_3 \bar{a}_3 a_2 \bar{a}_2 a_3 = 0, 
  \quad 
  \bar{a}_3 a_2 \bar{a}_2 a_3 \bar{a}_3 a_2 = 0, 
  \\ 
  (\bar{a}_2 a_3 + \bar{a}_3 a_2)^2 = 0,
  \quad 
 \bar{a}_1 a_1 + \bar{a}_3 a_2 + \bar{a}_2 a_3 
  + f(\bar{a}_2 a_3, \bar{a}_3 a_2) = 0,
 \\
  {a}_2 \bar{a}_2 + \bar{a}_4 {a}_5 = 0,  
  \quad 
  {a}_3 \bar{a}_3 + \bar{a}_5 {a}_4 = 0,  
  \quad 
  {a}_4 \bar{a}_4 = 0,  
  \quad 
  {a}_5 \bar{a}_5 = 0.  
\end{gather*}
Then $\Lambda^f(\FF_4)$ is called a \emph{deformed mesh algebra}
of type ${\FF_4}$ (see \cite[Section~9]{ESk}).
Observe that $\Lambda^f(\FF_4)$ is obtained from $\Lambda(\FF_4)$
by deforming the relation at the exceptional vertex $0$ of $Q_{\FF_4}$,
and $\Lambda^f(\FF_4) = \Lambda(\FF_4)$ if $f = 0$.

The following lemma describes the structure of
admissible elements of ${\rm rad}^2R(\FF_4)$.

\begin{lemma*}
\label{lem}
An element $f$ from ${\rm rad}^2R(\FF_4)$ is admissible if and only if
\begin{align*}
 f(x,y)
  &=
  \theta_1
  xy
  +
  \theta_2
  yx
  +
  \theta_3
  xxx
  +
  \theta_4
  xxy
  +
  \theta_5
  yxx
  +
  \theta_6
  xxxx
  +
  \theta_7
  xxxy
 \\&
 \quad
  +
  \theta_8
  xxxxx
\end{align*}
for some $\theta_1, \dots, \theta_8 \in K$,
satisfying
$\theta_2 = - \theta_1$,
$\theta_5 = 2\theta_3-\theta_4+\theta_1^2$,
and
$2\theta_7 = 2(\theta_6-\theta_1^3+2\theta_1(\theta_4-\theta_3))$.
\end{lemma*}

\begin{proof}
We claim, that
$B = \{1_K,x,y,
  xx,
  xy,
  yx,
  xxx,
  xxy,
  yxx,
  xxxx,
  xxxy,
  xxxxx\}$
form a basis of $R(\FF_4)$ over $K$.
Indeed, it is easy to see, by induction on the degree of elements from $R(\FF_4)$,
that each element $\omega \in R(\FF_4)$
which is a multiplication of elements $x$ and $y$
is a  linear combination (possibly trivial)
of elements from $B$.
In particular we have:
\begin{align*}
  yy &= yy - (x+y)^2 = - ( xx + xy + yx ), \\
  xyy &= - x ( xx + xy + yx ) = - ( xxx + xxy ) - xyx = - ( xxx + xxy ) , \\
  yyx &= - ( xx + xy + yx ) x = - ( xxx + yxx ) , \\
  yyy &= - y ( xx + xy + yx ) = - ( yxx + yyx ) = - (yxx - ( xxx + yxx )) = xxx , \\
 xyyy &= -x(yxx+yyx) = -xyyx = x(xxx+yxx) = xxxx, \\
  yyyx &= - (xyy+xxy)x = - xyyx = xxxx , \\
  yyyy &= yxxx = - y(xyy+xxy) = - yxxy = (xxx+yyx)y = xxxy , \\
  xxyy &= -x(xxx + xxy) = -(xxxx + xxxy) , \\
  yyxx &= -(xxx + yxx)x = -(xxxx + yxxx) = -(xxxx + xxxy) . 
\end{align*}
Similarly, using above equations we obtain equalities:
\begin{align*}
  yyyyy = yyxxx = yxxxy = xxxyy = xyyxx = xxyyx &= - xxxxx, \\
  yyxxy = yxxyy = yyyxx = xyyyx = xxyyy &= xxxxx .
\end{align*}
Moreover, observe that
\begin{align*}
  x(xxxxx) = (xxxxx)x &= (yyxxy)x = yyx(xyx) = 0,
\\
  (xxxxx)y &= (xyyyx)y = xyy(yxy) = 0, 
\\
  y(xxxxx) &= y(xyyyx) = (yxy)yyx = 0, 
\end{align*}
and hence ${\rm rad}^6R(\FF_4) = 0$.

Let $f \in {\rm rad}^2R(\FF_4)$.
Then
\begin{align*}
 f(x,y)
  &=
  \theta_0
  xx
  +
  \theta_1
  xy
  +
  \theta_2
  yx
  +
  \theta_3
  xxx
  +
  \theta_4
  xxy
  +
  \theta_5
  yxx
  +
  \theta_6
  xxxx
 \\&
 \quad
  +
  \theta_7
  xxxy
  +
  \theta_8
  xxxxx
\end{align*}
for some $\theta_0, \dots, \theta_8 \in K$.
Then we have
\begin{align*}
\big(x+y + f(x,y)\big)^2
  &=
 \big(x+y
  +
  \theta_0
  xx
  +
  \theta_1
  xy
  +
  \theta_2
  yx
  +
  \theta_3
  xxx
  +
  \theta_4
  xxy
\\&\ \ \quad
  +
  \theta_5
  yxx
  +
  \theta_6
  xxxx
  +
  \theta_7
  xxxy
 \big)^2
 \\
  &=
(2\theta_0-\theta_1-\theta_2) xxx
+ \theta_0 xxy
+ \theta_0 yxx
\\&\quad
+ (2\theta_3-\theta_4-\theta_5+\theta_0^2-\theta_1 \theta_2) xxxx
\\&\quad
+ (2\theta_3-\theta_4-\theta_5+\theta_0 \theta_1+\theta_0 \theta_2-\theta_1 \theta_2)xxxy
\\&\quad
+ (2\theta_6-2\theta_7+2\theta_0 \theta_3-2\theta_1\theta_5-2\theta_2\theta_4)xxxxx
.
\end{align*}
Hence $f$ is admissible if and only if
there are satisfied equalities
\begin{align*}
 \theta_0 &= 0, \\
 2\theta_0-\theta_1-\theta_2 &= 0,\\
2\theta_3-\theta_4-\theta_5+\theta_0^2-\theta_1 \theta_2 &= 0,\\
2\theta_3-\theta_4-\theta_5+\theta_0 \theta_1+\theta_0 \theta_2-\theta_1 \theta_2 &= 0,\\
2\theta_6-2\theta_7+2\theta_0 \theta_3-2\theta_1\theta_5-2\theta_2\theta_4&= 0. 
\end{align*}
Clearly, these equalities are equivalent to the equalities
\begin{align*}
 \theta_0 &= 0, \\
 \theta_2 &= -\theta_1,\\
 \theta_5  &= 2\theta_3-\theta_4+\theta_1^2 ,\\
2\theta_7 &= 2\big(\theta_6-\theta_1^3+2\theta_1(\theta_4-\theta_3)\big) .
\end{align*}
This ends the proof.
\end{proof}

The remaining part of this article 
is devoted to the proof of 
Theorem~\ref{thm}.

\medskip

Let $f$ be an admissible element of ${\rm rad}^2 R(\FF_4)$.
We will show that the algebras
$\Lambda(\FF_4)$ and $\Lambda^f(\FF_4)$
are isomorphic.
This will be done via a change of generators in $\Lambda(\FF_4)$.
It follows 
from 
Lemma~\ref{lem}
that
there exist $\theta_1, \dots, \theta_8 \in K$,
satisfying
$\theta_2 = - \theta_1$,
$\theta_5 = 2\theta_3-\theta_4+\theta_1^2$,
and
$2\theta_7 = 2(\theta_6-\theta_1^3+2\theta_1(\theta_4-\theta_3))$
such that
\begin{align*}
 f(x,y)
  &=
  \theta_1
  xy
  +
  \theta_2
  yx
  +
  \theta_3
  xxx
  +
  \theta_4
  xxy
  +
  \theta_5
  yxx
  +
  \theta_6
  xxxx
  +
  \theta_7
  xxxy
 \\&
 \quad
  +
  \theta_8
  xxxxx .
\end{align*}

Now we change generators in $\Lambda(\FF_4)$.
We replace 
$a_1$ by $a_1' \in \Lambda(\FF_4)$ 
and 
$\bar{a}_i$ by $\bar{a}_i' \in \Lambda(\FF_4)$,
for $i \in \{1,3\}$, 
defined as follows
\begin{align*}
  a'_1 &= a_1
-\theta_1 a_1 \bar{a}_2 a_3
+(\theta_4-\theta_3-\theta_1^2)
a_1 \bar{a}_2 a_3 \bar{a}_3 a_2 ,
\\
  \bar{a}'_1 &= \bar{a}_1
+ \theta_1 \bar{a}_2 a_3 \bar{a}_1
+ (\theta_1^2 + \theta_3)
\bar{a}_2 a_3 \bar{a}_2 a_3 \bar{a}_1
\\&\quad
+ (2\theta_3-\theta_4+\theta_1^2)
\bar{a}_3 a_2 \bar{a}_2 a_3 \bar{a}_1
+ (\theta_6+\theta_1^3+\theta_1\theta_3)
\bar{a}_2 a_3 \bar{a}_2 a_3 \bar{a}_2 a_3 \bar{a}_1
\\&\quad
+ (\theta_1\theta_6+2\theta_1^2\theta_3+2\theta_3^2+3\theta_3\theta_4+2\theta_1^2\theta_4-\theta_4^2)
\bar{a}_2 a_3 \bar{a}_2 a_3 \bar{a}_2 a_3 \bar{a}_2 a_3 \bar{a}_1
,\\
  \bar{a}'_3 &= \bar{a}_3
+
(\theta_6-\theta_1^3+2\theta_1\theta_4-2\theta_1 \theta_3-\theta_7)
\bar{a}_2 a_3 \bar{a}_2 a_3 \bar{a}_2 a_3 \bar{a}_3
\\&\quad
+
\theta_8
\bar{a}_2 a_3 \bar{a}_2 a_3 \bar{a}_2 a_3 \bar{a}_3 a_2 \bar{a}_3 
\end{align*}
and keep all other arrows 
(i.e. $a_2, a_3, a_4, a_5, \bar{a}_2, \bar{a}_4, \bar{a}_5 $) 
as they are.
Then
\begin{align*}
  a_1 &= a'_1
+\theta_1 a'_1 \bar{a}_2 a_3
+\theta_1^2 a'_1 \bar{a}_2 a_3 \bar{a}_2 a_3
-(\theta_4-\theta_3-\theta_1^2)
a'_1 \bar{a}_2 a_3 \bar{a}'_3 a_2 ,
\\&\quad
+\theta_1^3 a'_1 \bar{a}_2 a_3 \bar{a}_2 a_3 \bar{a}_2 a_3
+\theta_1^4 a'_1 \bar{a}_2 a_3 \bar{a}_2 a_3 \bar{a}_2 a_3 \bar{a}_2 a_3
\\
  \bar{a}_1 &= \bar{a}'_1
- \theta_1 \bar{a}_2 a_3 \bar{a}'_1
- \theta_3
\bar{a}_2 a_3 \bar{a}_2 a_3 \bar{a}'_1
- (2\theta_3-\theta_4+\theta_1^2)
\bar{a}'_3 a_2 \bar{a}_2 a_3 \bar{a}'_1
,
\\&\quad
+ (\theta_1\theta_3-\theta_6)
\bar{a}_2 a_3 \bar{a}_2 a_3 \bar{a}_2 a_3 \bar{a}'_1
\\&\quad
+ (\theta_1\theta_6+\theta_1^2\theta_3-\theta_3^2-3\theta_3\theta_4-2\theta_1^2\theta_4+\theta_4^2+\theta_2^4)
\bar{a}_2 a_3 \bar{a}_2 a_3 \bar{a}_2 a_3 \bar{a}_2 a_3 \bar{a}'_1
,
\\
  \bar{a}_3 &= \bar{a}'_3
-
(\theta_6-\theta_1^3+2\theta_1\theta_4-2\theta_1 \theta_3-\theta_7)
\bar{a}_2 a_3 \bar{a}_2 a_3 \bar{a}_2 a_3 \bar{a}'_3
\\&\quad
-
\theta_8
\bar{a}_2 a_3 \bar{a}_2 a_3 \bar{a}_2 a_3 \bar{a}'_3 a_2 \bar{a}'_3 .
\end{align*}
Therefore this is an invertible change 
of generators.

We will show now that, with these new generators, 
$\Lambda(\FF_4)$ satisfies the relations of $\Lambda^f(\FF_4)$.
Hence we need to show the equalities
\begin{gather*}
  a'_1 \bar{a}'_1 = 0, 
  \quad 
  {a}_3 \bar{a}'_3 + \bar{a}_5 {a}_4 = 0,  
  \quad 
 \bar{a}'_1 a'_1 + \bar{a}'_3 a_2 + \bar{a}_2 a_3 
  + f(\bar{a}_2 a_3, \bar{a}'_3 a_2) = 0.
\end{gather*}

We note  that
from the relations
$a_1 \bar{a}_1 = 0$
and
$\bar{a}_1 a_1 + \bar{a}_2 a_3 + \bar{a}_3 a_2 = 0$
we obtain the equalities
\begin{align*}
a_1 \bar{a}_2 a_3 \bar{a}_3 a_2 \bar{a}_1
 &=
  - a_1 \bar{a}_2 a_3 \bar{a}_2 a_3 \bar{a}_1
  - a_1 \bar{a}_2 a_3 \bar{a}_1 a_1 \bar{a}_1
 =
  - a_1 \bar{a}_2 a_3 \bar{a}_2 a_3 \bar{a}_1
,\\
a_1 \bar{a}_3 a_2 \bar{a}_2 a_3 \bar{a}_1
 &=
  - a_1 \bar{a}_2 a_3 \bar{a}_2 a_3 \bar{a}_1
  - a_1 \bar{a}_1 a_1 \bar{a}_2 a_3 \bar{a}_1
 =
  - a_1 \bar{a}_2 a_3 \bar{a}_2 a_3 \bar{a}_1
,\\
a_1 \bar{a}_3 a_2 \bar{a}_3 a_2 \bar{a}_1
 &=
  - a_1 \bar{a}_2 a_3 \bar{a}_3 a_2 \bar{a}_1
  - a_1 \bar{a}_1 a_1 \bar{a}_3 a_2 \bar{a}_1
 =
  a_1 \bar{a}_2 a_3 \bar{a}_2 a_3 \bar{a}_1
.
\end{align*}
Note also that from the calculation from the proof of
Lemma~\ref{lem}
(with $x=\bar{a}_2 a_3$ and $y =\bar{a}_3 a_2$)
we have the equality
$\bar{a}_2 a_3 \bar{a}_3 a_2 \bar{a}_3 a_2 \bar{a}_2 a_3
= - \bar{a}_2 a_3 \bar{a}_2 a_3 \bar{a}_2 a_3 \bar{a}_2 a_3$,
and hence the equality
$$a_1 \bar{a}_2 a_3 \bar{a}_3 a_2 \bar{a}_3 a_2 \bar{a}_2 a_3  \bar{a}_1
= - a_1 \bar{a}_2 a_3 \bar{a}_2 a_3 \bar{a}_2 a_3 \bar{a}_2 a_3 \bar{a}_1 .$$
Moreover, we have the equalities
\begin{align*}
  (\theta_1\theta_6+2\theta_1^2\theta_3&+2\theta_3^2+3\theta_3\theta_4+2\theta_1^2\theta_4-\theta_4^2)
  - \theta_1 (\theta_6+\theta_1^3+\theta_1\theta_3)
\\&\qquad
  - (\theta_4-\theta_3-\theta_1^2) (2\theta_3-\theta_4+\theta_1^2)
\\&=
  (\theta_1^2\theta_3+2\theta_3^2+3\theta_3\theta_4+2\theta_1^2\theta_4-\theta_4^2-\theta_2^4)
\\&\qquad
  - (3\theta_3\theta_4-\theta_4^2+2\theta_1^2\theta_4-2\theta_3^2-3\theta_1^2\theta_3-\theta_1^4)
  = 0
.
\end{align*}
Then we obtain
\begin{align*}
a'_1 \bar{a}'_1
 &= a_1 \bar{a}_1
 + (\theta_1-\theta_1) a_1 \bar{a}_2 a_3 \bar{a}_1
 + \big(\theta_1^2
   + (\theta_6+\theta_1^3+\theta_1\theta_3)
\\&\qquad
   - (\theta_4-\theta_3-\theta_1^2)
   - (2\theta_3-\theta_4+\theta_1^2)
     \big)
   a_1 \bar{a}_2 a_3 \bar{a}_2 a_3 \bar{a}_1
\\&\quad
 + \big(
  (\theta_1\theta_6+2\theta_1^2\theta_3+2\theta_3^2+3\theta_3\theta_4+2\theta_1^2\theta_4-\theta_4^2)
  - \theta_1 (\theta_6+\theta_1^3+\theta_1\theta_3)
\\&\qquad
  - (\theta_4-\theta_3-\theta_1^2) (2\theta_3-\theta_4+\theta_1^2)
   \big)
  a_1 \bar{a}_2 a_3 \bar{a}_2 a_3 \bar{a}_2 a_3 \bar{a}_2 a_3 \bar{a}_1
 = 0 .
\end{align*}
We have also
$$
 a_3 \bar{a}_2 a_3 \bar{a}_2 a_3 \bar{a}_2 a_3 \bar{a}_3
 =
 - a_3 (\bar{a}_3 a_2 \bar{a}_3 a_2 \bar{a}_2 a_3 + \bar{a}_3 a_2 \bar{a}_2 a_3 \bar{a}_2 a_3) \bar{a}_3
 = 0 ,
$$
and similarly
$$
 a_3 \bar{a}_2 a_3 \bar{a}_2 a_3 \bar{a}_2 a_3 \bar{a}_2 a_3 \bar{a}_3
 = 0 ,
$$
and hence we obtain
\begin{align*}
a_3 \bar{a}'_3
+ \bar{a}_5 a_4
&=
 a_3 \bar{a}_3 + \bar{a}_5 a_4
 = 0.
\end{align*}
Further, observe that
\begin{align*}
\bar{a}'_3 a_2
  &= \bar{a}_3 a_2
+
(\theta_6-\theta_1^3+2\theta_1\theta_4-2\theta_1 \theta_3-\theta_7)
\bar{a}_2 a_3 \bar{a}_2 a_3 \bar{a}_2 a_3 \bar{a}_3 a_2
\\&\quad
+
\theta_8
\bar{a}_2 a_3 \bar{a}_2 a_3 \bar{a}_2 a_3 \bar{a}_3 a_2 \bar{a}_3 a_2 
\\
 &= \bar{a}_3 a_2
+
(\theta_6-\theta_1^3+2\theta_1\theta_4-2\theta_1 \theta_3-\theta_7)
\bar{a}_2 a_3 \bar{a}_2 a_3 \bar{a}_2 a_3 \bar{a}_3 a_2
\\&\quad
-
\theta_8
\bar{a}_2 a_3 \bar{a}_2 a_3 \bar{a}_2 a_3 \bar{a}_2 a_3 \bar{a}_2 a_3
\end{align*}
and
\begin{align*}
f (\bar{a}_2 a_3, \bar{a}'_3 a_2) 
&
  = f(\bar{a}_2 a_3, \bar{a}_3 a_2)
\\&\qquad
  +
\theta_1(\theta_6-\theta_1^3+2\theta_1\theta_4-2\theta_1 \theta_3-\theta_7)
\bar{a}_2 a_3 \bar{a}_2 a_3 \bar{a}_2 a_3 \bar{a}_2 a_3 \bar{a}_3 a_2
\\&\qquad
  +
\theta_2(\theta_6-\theta_1^3+2\theta_1\theta_4-2\theta_1 \theta_3-\theta_7)
\bar{a}_2 a_3 \bar{a}_2 a_3 \bar{a}_2 a_3 \bar{a}_3 a_2 \bar{a}_2 a_3
\\&
  = f(\bar{a}_2 a_3, \bar{a}_3 a_2)
.
\end{align*}
In the calculations of
$\bar{a}'_1 a'_1$
we will use
the formulas derived in the proof of
Lemma~\ref{lem}
(with $x=\bar{a}_2 a_3$ and $y =\bar{a}_3 a_2$)
and
the substitution
$\bar{a}_1 a_1 = - (\bar{a}_2 a_3 + \bar{a}_3 a_2)$.
We have
\begin{align*}
\bar{a}'_1 a'_1
  &= \bar{a}_1 a_1
-\theta_1 \bar{a}_1 a_1 \bar{a}_2 a_3
+(\theta_4-\theta_3-\theta_1^2)
\bar{a}_1 a_1 \bar{a}_2 a_3 \bar{a}_3 a_2
\\&\quad
+  \theta_1 \bar{a}_2 a_3 \bar{a}_1 a_1
-\theta_1^2 \bar{a}_2 a_3 \bar{a}_1 a_1 \bar{a}_2 a_3
\\&\quad
+\theta_1 (\theta_4-\theta_3-\theta_1^2)
 \bar{a}_2 a_3 \bar{a}_1 a_1 \bar{a}_2 a_3 \bar{a}_3 a_2
\\&\quad
+ (\theta_1^2 + \theta_3)
\bar{a}_2 a_3 \bar{a}_2 a_3 \bar{a}_1 a_1
-\theta_1 (\theta_1^2 + \theta_3)
\bar{a}_2 a_3 \bar{a}_2 a_3 \bar{a}_1 a_1 \bar{a}_2 a_3
\\&\quad
+(\theta_4-\theta_3-\theta_1^2)
(\theta_1^2 + \theta_3)
\bar{a}_2 a_3 \bar{a}_2 a_3 \bar{a}_1 a_1 \bar{a}_2 a_3 \bar{a}_3 a_2
\\&\quad
+ (2\theta_3-\theta_4+\theta_1^2)
\bar{a}_3 a_2 \bar{a}_2 a_3 \bar{a}_1 a_1
\\&\quad
-\theta_1 (2\theta_3-\theta_4+\theta_1^2)
\bar{a}_3 a_2 \bar{a}_2 a_3 \bar{a}_1 a_1 \bar{a}_2 a_3
\\&\quad
+(\theta_4-\theta_3-\theta_1^2)
(2\theta_3-\theta_4+\theta_1^2)
\bar{a}_3 a_2 \bar{a}_2 a_3 \bar{a}_1 a_1 \bar{a}_2 a_3 \bar{a}_3 a_2
\\&\quad
+ (\theta_6+\theta_1^3+\theta_1\theta_3)
\bar{a}_2 a_3 \bar{a}_2 a_3 \bar{a}_2 a_3 \bar{a}_1 a_1
\\&\quad
- \theta_1 (\theta_6+\theta_1^3+\theta_1\theta_3)
\bar{a}_2 a_3 \bar{a}_2 a_3 \bar{a}_2 a_3 \bar{a}_1 a_1 \bar{a}_2 a_3
\\&\quad
+(\theta_1\theta_6+2\theta_1^2\theta_3+2\theta_3^2+3\theta_3\theta_4+2\theta_1^2\theta_4-\theta_4^2)
\bar{a}_2 a_3 \bar{a}_2 a_3 \bar{a}_2 a_3 \bar{a}_2 a_3 \bar{a}_1 a_1
%
\end{align*}
\begin{align*}
  &= \bar{a}_1 a_1
+\theta_1 ( \bar{a}_2 a_3 + \bar{a}_3 a_2 ) \bar{a}_2 a_3
-(\theta_4-\theta_3-\theta_1^2)
\bar{a}_2 a_3 \bar{a}_2 a_3 \bar{a}_3 a_2
\\&\quad
-  \theta_1 \bar{a}_2 a_3 ( \bar{a}_2 a_3 + \bar{a}_3 a_2 )
+\theta_1^2 \bar{a}_2 a_3 \bar{a}_2 a_3 \bar{a}_2 a_3
\\&\quad
-\theta_1 (\theta_4-\theta_3-\theta_1^2)
 \bar{a}_2 a_3 \bar{a}_2 a_3 \bar{a}_2 a_3 \bar{a}_3 a_2
\\&\quad
    - (\theta_1^2 + \theta_3)
\bar{a}_2 a_3 \bar{a}_2 a_3 ( \bar{a}_2 a_3 + \bar{a}_3 a_2 )
+\theta_1 (\theta_1^2 + \theta_3)
\bar{a}_2 a_3 \bar{a}_2 a_3 \bar{a}_2 a_3 \bar{a}_2 a_3
\\&\quad
- (2\theta_3-\theta_4+\theta_1^2)
\bar{a}_3 a_2 \bar{a}_2 a_3 \bar{a}_2 a_3
\\&\quad
+\theta_1 (2\theta_3-\theta_4+\theta_1^2)
\bar{a}_3 a_2 \bar{a}_2 a_3 \bar{a}_2 a_3 \bar{a}_2 a_3
\\&\quad
-(\theta_4-\theta_3-\theta_1^2)
(2\theta_3-\theta_4+\theta_1^2)
\bar{a}_3 a_2 \bar{a}_2 a_3 \bar{a}_2 a_3 \bar{a}_2 a_3 \bar{a}_3 a_2
\\&\quad
- (\theta_6+\theta_1^3+\theta_1\theta_3)
\bar{a}_2 a_3 \bar{a}_2 a_3 \bar{a}_2 a_3 ( \bar{a}_2 a_3 + \bar{a}_3 a_2 )
\\&\quad
+ \theta_1 (\theta_6+\theta_1^3+\theta_1\theta_3)
\bar{a}_2 a_3 \bar{a}_2 a_3 \bar{a}_2 a_3 \bar{a}_2 a_3 \bar{a}_2 a_3
\\&\quad
-(\theta_1\theta_6+2\theta_1^2\theta_3+2\theta_3^2+3\theta_3\theta_4+2\theta_1^2\theta_4-\theta_4^2)
\bar{a}_2 a_3 \bar{a}_2 a_3 \bar{a}_2 a_3 \bar{a}_2 a_3 \bar{a}_2 a_3
\\
  &= \bar{a}_1 a_1
+ (\theta_1 -\theta_1) \bar{a}_2 a_3 \bar{a}_2 a_3
- \theta_1 \bar{a}_2 a_3 \bar{a}_3 a_2
+ \theta_1 \bar{a}_3 a_2 \bar{a}_2 a_3
\\&\quad
+\big(\theta_1^2 - (\theta_1^2 + \theta_3)\big) \bar{a}_2 a_3 \bar{a}_2 a_3 \bar{a}_2 a_3
\\&\quad
- \big((\theta_4-\theta_3-\theta_1^2) + (\theta_1^2 + \theta_3)\big)
\bar{a}_2 a_3 \bar{a}_2 a_3 \bar{a}_3 a_2
\\&\quad
- (2\theta_3-\theta_4+\theta_1^2)
\bar{a}_3 a_2 \bar{a}_2 a_3 \bar{a}_2 a_3
\\&\quad
+\big(\theta_1 (\theta_1^2 + \theta_3)
- (\theta_6+\theta_1^3+\theta_1\theta_3)\big)
\bar{a}_2 a_3 \bar{a}_2 a_3 \bar{a}_2 a_3 \bar{a}_2 a_3
\\&\quad
+\big(\theta_1 (2\theta_3-\theta_4+\theta_1^2)
- \theta_1 (\theta_4-\theta_3-\theta_1^2)
\\&\ \ \qquad
- (\theta_6+\theta_1^3+\theta_1\theta_3)
\big)
 \bar{a}_2 a_3 \bar{a}_2 a_3 \bar{a}_2 a_3 \bar{a}_3 a_2
\\&\quad
+\big((\theta_4-\theta_3-\theta_1^2)
(2\theta_3-\theta_4+\theta_1^2)
+ \theta_1 (\theta_6+\theta_1^3+\theta_1\theta_3)
-(\theta_1\theta_6
\\&\ \ \qquad
+2\theta_1^2\theta_3+2\theta_3^2+3\theta_3\theta_4+2\theta_1^2\theta_4-\theta_4^2)
\big)
\bar{a}_2 a_3 \bar{a}_2 a_3 \bar{a}_2 a_3 \bar{a}_2 a_3 \bar{a}_2 a_3
\\
  &= \bar{a}_1 a_1
- \theta_1 \bar{a}_2 a_3 \bar{a}_3 a_2
+ \theta_1 \bar{a}_3 a_2 \bar{a}_2 a_3
- \theta_3 \bar{a}_2 a_3 \bar{a}_2 a_3 \bar{a}_2 a_3
\\&\quad
- \theta_4
\bar{a}_2 a_3 \bar{a}_2 a_3 \bar{a}_3 a_2
- \theta_5
\bar{a}_3 a_2 \bar{a}_2 a_3 \bar{a}_2 a_3
- \theta_6
\bar{a}_2 a_3 \bar{a}_2 a_3 \bar{a}_2 a_3 \bar{a}_2 a_3
\\&\quad
+(2\theta_1 \theta_3 - 2 \theta_1 \theta_4+\theta_1^3 - \theta_6)
 \bar{a}_2 a_3 \bar{a}_2 a_3 \bar{a}_2 a_3 \bar{a}_3 a_2
.
\end{align*}
Summing up the above equations, we obtain
\begin{align*}
 \bar{a}'_1 a'_1 + \bar{a}'_3 a_2 + \bar{a}_2 a_3 
  + f(\bar{a}_2 a_3, \bar{a}'_3 a_2) 
  = 
 \bar{a}_1 a_1 + \bar{a}_3 a_2 + \bar{a}_2 a_3 
= 0.
\end{align*}
Observe also that
\begin{align*}
  \bar{a}_2 a_3 \bar{a}'_3 a_2 \bar{a}_2 a_3 
  &= \bar{a}_2 a_3 \bar{a}_3 a_2 \bar{a}_2 a_3 
  = 0, 
\\ 
  \bar{a}'_3 a_2 \bar{a}_2 a_3 \bar{a}'_3 a_2 
  &= \bar{a}_3 a_2 \bar{a}_2 a_3 \bar{a}_3 a_2 
   = 0, 
\end{align*}
because $\rad^{11} \Lambda(\mathbb{F}_4) = 0$. 
Finally, applying the equalities
$(\bar{a}_2 a_3)^4 \bar{a}'_3 a_2 = 0$,
$(\bar{a}_2 a_3)^3 \bar{a}'_3 a_2 \bar{a}_2 a_3 = 0$,
and
$2(\theta_6-\theta_1^3+2\theta_1(\theta_4-\theta_3) - \theta_7) = 0$,
we obtain
\begin{align*}
  (\bar{a}_2 a_3 + \bar{a}'_3 a_2)^2 
  &= (\bar{a}_2 a_3 + \bar{a}_3 a_2)^2 
\\&\quad
   + (\theta_6-\theta_1^3+2\theta_1(\theta_4-\theta_3) - \theta_7) 
       (\bar{a}_2 a_3)^3 (\bar{a}_3 a_2)^2 
\\&\quad
   + (\theta_6-\theta_1^3+2\theta_1(\theta_4-\theta_3) - \theta_7) 
            + \bar{a}_3 a_2 (\bar{a}_2 a_3)^3 \bar{a}_3 a_2
\\&
  =
   (- \bar{a}_1 a_1)^2 
   -2 (\theta_6-\theta_1^3+2\theta_1(\theta_4-\theta_3) - \theta_7) 
       (\bar{a}_2 a_3)^5
\\&
   = 0.
\end{align*}
Hence with these new generators, 
$\Lambda(\FF_4)$ satisfies the relations of $\Lambda^f(\FF_4)$,
and consequently 
the algebras $\Lambda(\FF_4)$ and $\Lambda^f(\FF_4)$
are isomorphic.

%

\end{document}